\documentclass{amsart}

\usepackage[english]{babel}

\usepackage{amsmath, amsfonts, amssymb, amsthm, faktor}

\usepackage[all]{xy}
\usepackage{tikz}
\usetikzlibrary{math}
\usetikzlibrary{patterns}
\usepackage{caption}
\usepackage{subfig}

\usetikzlibrary{arrows,chains,matrix,positioning,scopes,decorations.pathreplacing,
decorations.pathmorphing,decorations.markings,arrows.meta}

\usepackage{aliascnt}
\usepackage[colorlinks, linktocpage, allcolors=black,breaklinks]{hyperref}

\usepackage{enumerate}
\usepackage{verbatim}

\theoremstyle{definition}
\newtheorem{theorem}{Theorem}[section]

\newtheorem{corollary}[theorem]{Corollary}
\newtheorem{lemma}[theorem]{Lemma}

\newtheorem*{remark}{Remark}

\newcommand{\Z}{\mathbb Z}
\newcommand{\R}{\mathbb R}

\newcommand{\norm}[1]{\lVert#1\rVert}
\DeclareMathOperator{\cone}{cone}
\DeclareMathOperator{\Hull}{Hull}
\DeclareMathOperator{\proj}{proj}

\begin{document}

\title[Strong Marker Sets for Arbitrary Generating Sets]{Strong Marker Sets for Arbitrary Generating Sets}

\author{Su Gao}
\address{School of Mathematical Sciences and LPMC, Nankai University, Tianjin 300071, P.R. China}
\email{sgao@nankai.edu.cn}
\thanks{The authors acknowledge the partial support of their research by the Fundamental Research Funds for the Central Universities and by the National Natural Science Foundation of China (NSFC) grant 12271263.}

\author{Tianhao Wang}
\address{School of Mathematics and Computer Science, Quanzhou Normal University, Quanzhou 362000, P.R. China}
\email{tianhao\_wang@qq.com}

\date{\today}

\begin{abstract}
We prove the existence of clopen strong marker sets in $F(2^{\mathbb{Z}^n})$
for arbitrary finite generating sets. Specifically, for any positive integers
$n, d_0\geq 1$ and any finite generating set $S\subseteq \mathbb{Z}^n$, we construct
a clopen set $M\subseteq F(2^{\mathbb{Z}^n})$ and a positive integer $\Delta$ such that
\begin{enumerate}
\item[(1)] for any distinct $x,y\in M$ in the same orbit,
$\rho(x,y)\geq d_0$;
\item[(2)] for any $v\in S$ and any $x\in F(2^{\mathbb{Z}^n})$, there are
non-negative integers $a,b\leq \Delta$ such that
$av\cdot x\in M$ and $-bv\cdot x\in M$.
\end{enumerate}
Here $\rho$ denotes the Euclidean metric. The same result then holds for the
standard supremum-norm metric $\rho_\infty$ (with an adjusted constant), by the
equivalence of norms on $\mathbb{Z}^n$. The proof introduces polyhedral packages
in $\mathbb{R}^n$ as a generalization of the rectangular packages used in earlier
work, enabling the construction to handle generating vectors with arbitrary
coordinate patterns. As an application, we obtain a continuous proper edge
$(2|S|+1)$-coloring of the Schreier graph on $F(2^{\mathbb{Z}^n})$ with generating
set $S$, recovering a result of Gao--Wang--Wang.
\end{abstract}

\maketitle

\section{Introduction}

The construction of marker sets is a fundamental tool in the descriptive set
theory of countable Borel equivalence relations. What are now called
{\em marker sets} were previously referred to as {\em cross sections}
(e.g.\ in \cite{We84}) or {\em complete sections} (e.g.\ in \cite{DJK94}).
The modern terminology, emphasizing the geometric nature of the orbits,
originates from \cite{JKL02}, though the concept was studied much earlier
(see \cite{SS88} and \cite{BK96}).

The typical framework for our study is the Bernoulli shift action of a
finitely generated group on a zero-dimensional Polish space. More concretely,
for a countable group $\Gamma$, the {\em Bernoulli shift action} of $\Gamma$
on $(2^{\mathbb{N}})^\Gamma$ is defined by
\[
(g\cdot x)(h)=x(hg),\qquad g,h\in\Gamma,\; x\in (2^{\mathbb{N}})^\Gamma.
\]
By a theorem of Jackson--Kechris--Louveau \cite[Proposition 4.2]{JKL02},
every Borel action of $\Gamma$ can be equivariantly embedded into this single
action. When the action is free, each orbit inherits the algebraic structure
of $\Gamma$. If $\Gamma$ is finitely generated with a finite generating
set $S$, the Cayley graph of $(\Gamma, S)$ transfers to each orbit, yielding
the {\em Schreier graph} of the action.

The simplest and most important space to consider is $F(2^\Gamma)$, the free
part of the Bernoulli shift action of $\Gamma$ with alphabet $\{0,1\}$.
The study of marker sets in $F(2^{\mathbb{Z}^n})$ was initiated by Gao and
Jackson in \cite{GJ15}, where they proved the following fundamental result.

\begin{lemma}[Basic clopen marker lemma {\cite[Lemma 2.1]{GJ15}}]
\label{lem:basicmarker-intro}
For any positive integer $d\geq 1$, there is a clopen set
$M\subseteq F(2^{\mathbb{Z}^n})$ such that
\begin{enumerate}
\item[(1)] for any distinct $x, y\in M$ in the same orbit,
$\rho_\infty(x,y)\geq d$;
\item[(2)] for any $x\in F(2^{\mathbb{Z}^n})$ there is $y\in M$
with $\rho_\infty(x,y)<d$.
\end{enumerate}
\end{lemma}

They also proved a marker regions lemma \cite[Theorem 3.1]{GJ15}, which
partitions $F(2^{\mathbb{Z}^n})$ into clopen $n$-dimensional rectangles
of uniformly bounded side lengths.

In \cite{GJKS22} and \cite{GJKS23}, limitations of marker constructions
were established. For instance, \cite[Theorem 1.2]{GJKS22} shows that for
any countably infinite group $\Gamma$, there is no Borel marker set
$M\subseteq F(2^\Gamma)$ such that for any finite $F\subseteq \Gamma$
and any $x$, there is $y\in [x]$ with $F\cdot y\cap M=\varnothing$.

In a recent paper \cite{GW25}, the authors proved the existence of clopen
marker sets with a strong regularity property for the standard generating
set $\{e_1,\dots,e_n\}$ of $\mathbb{Z}^n$: for any integer $d\geq 2$, there exists an integer $D$ and a clopen subset $M\subseteq F(2^{\Z^n})$ such that 
\begin{enumerate}
\item[(1)] for any distinct $x,y\in M$ in the same orbit,
$\rho_\infty(x,y)\geq d$;
\item[(2)] for any $1\leq i\leq n$ and any $x\in F(2^{\mathbb{Z}^n})$,
there are non-negative integers $a,b\leq D$ such that
$ae_i\cdot x\in M$ and $-be_i\cdot x\in M$.
\end{enumerate}
Two applications were given: a continuous proper edge $(2n+1)$-coloring of
the Schreier graph, and a clopen tree section in $F(2^{\mathbb{Z}^n})$.
The main theorem of \cite{GW25} was also generalized later in that same
paper to generating sets $S\subseteq \mathbb{Z}^n$ where each $g\in S$ is
required to have either exactly $1$ or exactly $n$ non-zero coordinates.
In dimension $2$, this condition holds for any non-zero vector, so the
result covers all generating sets in $\mathbb{Z}^2$.

The purpose of the present paper is to remove all restrictions on the
generating set. Our main theorem is the following.

\begin{theorem}\label{thm:main}
Let $n, d_0\geq 1$ be positive integers and let $S\subseteq \mathbb{Z}^n$ be a finite generating set. Then there is a positive integer $\Delta$ and a clopen subset $M\subseteq F(2^{\mathbb{Z}^n})$ such that
\begin{enumerate}
\item[(1)] for any distinct $x,y\in M$ in the same orbit, $\rho(x,y)\geq d_0$;
\item[(2)] for any $v\in S$ and any $x\in F(2^{\mathbb{Z}^n})$, there are non-negative integers $a, b\leq \Delta$ such that $av\cdot x\in M$ and $-bv\cdot x\in M$.
\end{enumerate}
Here $\rho$ denotes the Euclidean metric on $\mathbb{Z}^n$ (see Section~\ref{sec:prelim}).
\end{theorem}

\begin{remark}\label{rem:infinity}
The standard metric used in descriptive set theory is the supremum-norm
metric $\rho_\infty(x,y)=\sup_i|x(i)-y(i)|$ (see Section~\ref{sec:prelim}).
Since $\rho_\infty\leq \rho\leq\sqrt{n}\,\rho_\infty$ on $\mathbb{Z}^n$,
Theorem~\ref{thm:main} immediately implies the same statement with
$\rho_\infty$ in place of $\rho$ and with $d_0$ replaced by
$\lceil d_0/\sqrt{n}\rceil$ (or, equivalently, one may apply
Theorem~\ref{thm:main} with Euclidean parameter $\sqrt{n}\,d_0$ to obtain
$\rho_\infty$-spacing at least $d_0$). All geometric conclusions---the
edge-coloring application, and the existence of tree sections---remain
valid for the standard supremum-norm metric.
\end{remark}

The main novelty in our proof is the use of {\em convex polyhedra} as
packages, replacing the rectangular packages used in \cite{GW25}. When a
generating vector $v=(\nu_1,\dots,\nu_n)$ has non-zero coordinates in several
directions, the natural geometry of the construction requires us to work with
parallelopipeds and more general polyhedra rather than axis-aligned rectangles.
The key technical lemmas (Lemmas~\ref{lem:layer}--\ref{lem:mainpoly}) develop
the machinery of slicing polyhedra into $(n-1)$-dimensional cross-sections
that serve as bases for these packages.

As an application, we recover the edge-coloring result
from \cite{GWW25}.

\begin{corollary}[{\cite[Theorem 1.1]{GWW25}}]\label{cor:coloring}
Let $n\geq 1$ and let $S\subseteq \mathbb{Z}^n$ be a finite generating set.
Then there is a continuous proper edge $(2|S|+1)$-coloring of the Schreier
graph on $F(2^{\mathbb{Z}^n})$ with generating set $S$.
\end{corollary}

After the completion of a draft of this paper we learned that Yu \cite{Y26} has given a different proof of Theorem \ref{thm:main} based on our earlier results in \cite{GW25}.

We would like to take this opportunity to remark that our Corollary \ref{cor:coloring} can be deduced from combining some earlier results of Bernshteyn \cite{B22,B23}. In fact, in \cite{B22} Bernshteyn showed that there is a $\mbox{\rm poly}(d, \log n)$ LOCAL algorithm for finding a $(d+1)$-edge-coloring of any Borel graph of maximum degree $d$ with $n$ vertices. Combined with the result of \cite{B23} that LOCAL algorithms yield continuous colorings, this implies the same bound for continuous edge-coloring of groups of subexponential growth. Since $\mathbb{Z}^n$ has polynomial growth, the above results apply to the action of $\mathbb{Z}^n$ on $F(2^{\mathbb{Z}^n})$ to give this result. We thank Andrew Marks for bringing our attention to this implication. 

The rest of the paper is organized as follows. In Section~\ref{sec:prelim}
we recall basic concepts concerning $\mathbb{Z}^n$, $F(2^{\mathbb{Z}^n})$,
and convex polyhedra in both $\R^n$ and $\Z^n$, and fix our conventions on notation. In
Section~\ref{sec:packages} we develop the theory of polyhedral packages,
culminating in the main technical lemmas. In Section~\ref{sec:proof}
we prove Theorem~\ref{thm:main}. In Section~\ref{sec:app} we give the
application to edge colorings.

\section{Preliminaries}\label{sec:prelim}

\subsection{Metric conventions and basic concepts}

Let $n\geq 1$ be an integer. For $x=(x_1,\dots, x_n)\in\mathbb{Z}^n$ and
$1\leq i\leq n$, let $x(i)=x_i$. For $1\leq i\leq n$, let
$\pi_i\colon \mathbb{Z}^n\to \mathbb{Z}$ be the projection
$\pi_i(x)=x(i)$ and $\sigma_i\colon \mathbb{Z}^n\to\mathbb{Z}^{n-1}$ the
projection
\[
\sigma_i(x)=(x(1),\dots, x(i-1), x(i+1), \dots, x(n)).
\]
For $1\leq i\leq n$, let $e_i\in\mathbb{Z}^n$ be the standard basis vector.

We will make use of two metrics on $\mathbb{Z}^n$. The {\em Euclidean metric}
({\em $\ell^2$ metric}) is
\[
\rho(x,y)=\sqrt{\sum_{i=1}^n (x(i)-y(i))^2}.
\]
The {\em supremum-norm metric} ({\em $\ell^\infty$ metric}) is
\[
\rho_\infty(x,y)=\sup\{|x(i)-y(i)|\colon 1\leq i\leq n\}.
\]
They satisfy
\[
\rho_\infty(x,y)\leq \rho(x,y)\leq \sqrt{n}\,\rho_\infty(x,y),\qquad
x,y\in\mathbb{Z}^n.
\]

Throughout this paper, unless explicitly stated otherwise, $\rho$ denotes
the Euclidean metric. We write $\|x\|=\rho(x,\overline{0})$ where
$\overline{0}=(0,\dots,0)\in\mathbb{Z}^n$. For $A, B\subseteq \mathbb{Z}^n$,
\[
\rho(A, B)=\inf\{\rho(x,y)\colon x\in A,\; y\in B\},
\]
and similarly $\rho(x,A)=\inf\{\rho(x,y)\colon y\in A\}$.
For subsets of $\mathbb{R}^n$, $\|\cdot\|_2$ denotes the standard
$\ell^2$-norm and $\rho(A,B)=\inf\{\|x-y\|_2\colon x\in A,\;y\in B\}$.
We also extend $\rho_\infty$ to $\mathbb{R}^n$ in the obvious way.
A {\em square neighborhood} always refers to the $\rho_\infty$-metric; the
term reflects the geometry of $\rho_\infty$-balls in $\mathbb{Z}^n$, which
are $n$-dimensional rectangles.

For integers $a<b$, $[a,b]=\{t\in\mathbb{Z}\colon a\leq t\leq b\}$ is an
{\em interval}. An {\em $n$-dimensional rectangle} is a set
\[
R=[a_1,b_1]\times \cdots\times[a_n, b_n],
\]
where $a_i<b_i$ for all $1\leq i\leq n$. Its {\em side length} in direction
$e_i$ is $b_i-a_i$. A {\em generalized $n$-dimensional rectangle} allows
$a_i\leq b_i$. The $2^n$ {\em extreme points} of $R$ are enumerated
canonically: for $1\leq k\leq 2^n$,
\[
x_k(i)=\begin{cases}
a_i, & \text{if the $i$-th least-significant digit of $k-1$ in
binary is $0$},\\
b_i, & \text{otherwise}.
\end{cases}
\]

For an $n$-dimensional rectangle $R\subseteq\mathbb{Z}^n$ and
$1\leq i\leq n$, let
\[
F_i^+(R)=\{x\in R\colon e_i\cdot x\notin R\},\qquad
F_i^-(R)=\{x\in R\colon -e_i\cdot x\notin R\},
\]
the {\em upper} and {\em lower faces} of $R$ in direction $e_i$.

\subsection{The space $F(2^{\mathbb{Z}^n})$}

The {\em Bernoulli shift action} of $\mathbb{Z}^n$ on $2^{\mathbb{Z}^n}$ is defined by 
$$(g\cdot x)(h)=x(g+h)$$ for $g,h\in\mathbb{Z}^n$, $x\in 2^{\mathbb{Z}^n}$.
The {\em free part} is
\[
F(2^{\mathbb{Z}^n})=\{x\in 2^{\mathbb{Z}^n}\colon
\forall g\in\mathbb{Z}^n\;(g\neq \overline{0}\rightarrow g\cdot x\neq x)\}.
\]
With the product topology, $2^{\mathbb{Z}^n}$ is a Polish space and
$F(2^{\mathbb{Z}^n})$ is a $G_\delta$ subspace and hence is itself a Polish space. The action is continuous
and free on $F(2^{\mathbb{Z}^n})$.

For $x\in F(2^{\mathbb{Z}^n})$, the {\em orbit} is $[x]=\mathbb{Z}^n\cdot x$.
By freeness, each orbit is a copy of $\mathbb{Z}^n$, so all geometric
concepts in $\mathbb{Z}^n$ transfer to orbits. In particular,
$\rho(x,y)=\|g\|$ where $g\in\mathbb{Z}^n$ is the unique element with
$g\cdot x=y$. Similarly, we speak of intervals, rectangles, and generalized
rectangles in an orbit.

For a finite generating set $S\subseteq \mathbb{Z}^n$, the {\em Schreier
graph} $G$ on $F(2^{\mathbb{Z}^n})$ has vertex set $F(2^{\mathbb{Z}^n})$
and edge set
\[
\{x,y\}\in E(G)\iff \exists v\in S\;(v\cdot x=y\text{ or }v\cdot y=x).
\]

\subsection{Polyhedra in $\mathbb{R}^n$ and $\mathbb{Z}^n$}\label{sec:polyhedra}

A {\em hyperplane} in $\mathbb{R}^n$ is an affine subspace
of dimension $n-1$. For such a plane $P$, a point $p\in P$, and a normal
vector $v$, the {\em closed half-space} defined by $(P,v)$ is
\[
\operatorname{HS}(P,v)=\{x\in\mathbb{R}^n\colon (x-p)\cdot v\leq 0\}.
\]
All hyperplanes that appear in this paper have integer normal vectors.

A {\em convex polyhedron} in $\mathbb{R}^n$ is a non-empty, bounded, convex
set that can be expressed as the intersection of finitely many closed
half-spaces. Throughout this paper, all polyhedra are convex, and we simply
call them polyhedra. A polyhedron representable as the intersection of $m$
closed half-spaces is called an {\em $\leq\!\! m$-hedron}. For half-spaces
$\operatorname{HS}(P_1,v_1),\dots,\operatorname{HS}(P_m,v_m)$, the resulting
polyhedron is denoted $$\operatorname{PH}((P_1,v_1),\dots,(P_m,v_m)).$$

For $1\leq j\leq m$, if
$P_j\cap \operatorname{PH}((P_1,v_1),\dots,(P_m,v_m))\neq\varnothing$,
the intersection is called a {\em face} of the polyhedron. A face of an
$n$-dimensional $\leq\!\! m$-hedron is an $\leq\!\! (m+1)$-hedron in $\R^n$ (by adding the corresponding half-space $\operatorname{HS}(P_j,-v_j)$).

Let $P=\operatorname{PH}((P_1,v_1),\dots,(P_m,v_m))$ be an $n$-dimensional
$\leq\!\!\! m$-hedron and let $\delta>0$ be a real number. The
{\em $\delta$-extension} of $P$ is
\[
\operatorname{PH}((P_1',v_1),\dots,(P_m',v_m)),
\]
where $P_i'$ is the hyperplane obtained by translating $P_i$
outward along its normal by Euclidean distance $\delta$. The
{\em $\delta$-core} of $P$ is the set
$\{x\in P\colon \rho(x, P^c)>\delta\}$, where
$P^c=\mathbb{R}^n\setminus P$.

For a vector $v=(\nu_1,\dots,\nu_n)\in\mathbb{R}^n$ and an index
$1\leq i\leq n$ with $\nu_i\neq 0$, we define constructions along direction
$v$ with respect to the $e_i$-coordinate. Let $B$ be an $(n-1)$-dimensional
polyhedron that is not parallel to $v$, and let $h>0$. The {\em $n$-dimensional polyhedron with base $B$,
direction $v$, and $e_i$-height $h$} is
\[
P(B, v, i, h)=\bigcup_{0\leq t\leq h/|\nu_i|} (B+tv).
\]
The {\em infinite generalized polyhedron} with base $B$ and direction $v$ is
\[
P(B, v)=B+\mathbb{R}v.
\]

We say that $x\in\mathbb{R}^n$ {\em crosses} $B$ {\em in direction $v$} if
the line $\{x+av\colon a\in\mathbb{R}\}$ is contained in $P(B,v)$, or equivalently, the line $\{x+av\colon a\in\mathbb{R}\}\cap B\neq\emptyset$. For an
$n$-dimensional polyhedron $H$, we say $x$ {\em crosses} $H$ {\em in
direction $v$} if there exists a face $F$ of $H$ such that $x$ crosses $F$
in direction $v$.

We call $P\subseteq\mathbb{Z}^n$ a {\em polyhedron} if there exists a polyhedron $P'\subseteq\mathbb{R}^n$ such that $P=P'\cap\mathbb{Z}^n$. For any set $A\subseteq\mathbb{Z}^n$,
$\Hull(A)$ denotes the convex hull of $A$ in $\mathbb{R}^n$. The dimension of $P$ in $\Z^n$ is the dimension of $\Hull(P)$ in $\R^n$. In this paper, for any polyhedron $P$ on $\Z^n$, its corresponding $P'\subseteq\R^n$ with $P=P'\cap\Z^n$ is explicit, so there is no ambiguity in calling $P$ an $\leq\!\! m$-hedron on $\Z^n$ if $P'$ is an $\leq\!\! m$-hedron on $\R^n$.

\subsection{Two known results}

The following results of Gao--Jackson \cite{GJ15} are fundamental
to our construction. They were originally proved for $\rho_\infty$; we restate them for $\rho$.

\begin{lemma}[Marker regions lemma {\cite[Theorem 3.1]{GJ15}}]\label{lem:markerregions}
Let $d\geq 1$ be a positive integer. Then there is a clopen equivalence
relation $E_d^n$ on $F(2^{\mathbb{Z}^n})$ such that each $E_d^n$-equivalence
class is an $n$-dimensional rectangle with side lengths either $d$ or $d+1$.
\end{lemma}

\begin{lemma}[Basic clopen marker lemma {\cite[Lemma 2.1]{GJ15}}]\label{lem:basicmarker}
For any positive integer $d\geq 1$, there is a clopen set
$M\subseteq F(2^{\mathbb{Z}^n})$ such that
\begin{enumerate}
\item[(1)] for any distinct $x, y\in M$ in the same orbit,
$\rho(x,y)\geq d$;
\item[(2)] for any $x\in F(2^{\mathbb{Z}^n})$ there is $y\in M$ with
$\rho(x,y)< \sqrt{n}\,d$.
\end{enumerate}
\end{lemma}

\section{Polyhedral packages in $\mathbb{Z}^n$}\label{sec:packages}

In this section we develop the technical machinery for constructing
marker sets using polyhedral packages. The key idea is that while the
rectangular packages of \cite{GW25} suffice for generators aligned with
coordinate axes, arbitrary generating vectors require packages whose
geometry reflects the direction of the generator. We begin with a
``slicing'' lemma that decomposes polyhedra into layers, then construct
marker sets in infinite strips, and finally prove the main packaging lemma.

\subsection{A slicing lemma}

\begin{lemma}\label{lem:layer}
Let $n, d$ be positive integers, $1\leq i\leq n$, and
$v=(\nu_1,\dots,\nu_n)\in\mathbb{R}^n$ with $\nu_i\neq 0$.
Let $F\subseteq\mathbb{R}^n$ be an $n$-dimensional polyhedron such that for
any $x\in F$, the projection of $\{x+av\colon a\in\mathbb{R}\}\cap F$
onto the $e_i$-coordinate has length at least $d$. Let $l=|\pi_i(F)|$.
Then there exists a finite sequence of $(n-1)$-dimensional polyhedra
$\{E_j\}_{1\leq j\leq \lceil l/d\rceil}$ such that:
\begin{enumerate}
\item[(1)] each $E_j$ has $e_i$ as a normal vector;
\item[(2)] for any $1\leq j\neq j'\leq \lceil l/d\rceil$,
$\rho(E_j, E_{j'})\geq d$;
\item[(3)] for any $x\in\mathbb{R}^n$ that crosses $F$ in direction $v$,
there exists $1\leq j\leq \lceil l/d\rceil$ such that $x$ crosses $E_j$
in direction $v$.
\end{enumerate}
\end{lemma}

\begin{proof}
Let $a=\min\pi_i(F)$. For each integer $j\geq 1$, define the
$(n-1)$-dimensional plane
\[
E_j=\mathbb{R}\times\cdots\times\mathbb{R}\times\{a+(j-1)d\}
\times\mathbb{R}\times\cdots\times\mathbb{R},
\]
where the fixed coordinate is in the $i$-th position. Then $E_j$ has normal
vector $e_i$. For $1\leq j\leq \lceil l/d\rceil$, we have
$E_j\cap F\neq\varnothing$.

Let $x\in\mathbb{R}^n$ cross $F$ in direction $v$. Then the line
$L=\{x+av\colon a\in\mathbb{R}\}$ intersects $F$ in a segment whose
endpoints are $y=x+bv$ and $y'=x+b'v$ with $b<b'$, such that
$x+cv\notin F$ for $c<b$ and $c>b'$. By hypothesis,
$\pi_i(y')-\pi_i(y)\geq d$. The function $t\mapsto \pi_i(x+tv)$ is
continuous, so $\pi_i(L\cap F)$ is an interval of length at least $d$.
Hence this interval must contain some point of the form $a+(j-1)d$,
which means $L$ intersects $E_j$ for the corresponding $j$. Thus $x$
crosses $E_j$ in direction $v$.

For distinct $j\neq j'$, the planes $E_j$ and $E_{j'}$ are parallel and
separated by distance $|(j-j')d|$ in the normal direction $e_i$,
so $\rho(E_j, E_{j'})\geq d$, verifying (2).
\end{proof}

\subsection{Marker sets in infinite strips}

The next lemma constructs marker sets in an infinite strip bounded only in
one coordinate direction.

\begin{lemma}\label{lem:generalH}
Let $n, d\geq 1$ be positive integers, $1\leq i\leq n$, and
$v=(\nu_1,\dots, \nu_n)\in\mathbb{Z}^n$ with $\nu_i\neq 0$.
Then there exists an integer $H(n, d, v, i)$ such that for any
\[
R=\mathbb{Z}\times\cdots\times\mathbb{Z}\times
\underbrace{\{h, h+1, \dots, h+H(n, d, v, i)\}}_{\text{$i$-th coordinate}}
\times\mathbb{Z}\times\cdots\times\mathbb{Z}
\]
with $h\in\mathbb{Z}$, there exists a subset $M\subseteq R$ such that
\begin{enumerate}
\item[(1)] for any distinct $x,y\in M$, $\rho(x,y)\geq d$;
\item[(2)] for any $x\in \mathbb{Z}^n$, there exists $a\in\mathbb{Z}$
such that $x+av\in M$.
\end{enumerate}
\end{lemma}

\begin{proof}
Without loss of generality assume $i=n$ and $h=0$.

We first construct a fundamental block. For
$(a_1,\dots, a_{n-1})\in \mathbb{Z}^{n-1}$, let
\[
S_{(a_1,\dots, a_{n-1})}=[a_1d, (a_1+1)d]\times\cdots
\times[a_{n-1}d, (a_{n-1}+1)d]\times\{0\}.
\]
Define the finite parallelopiped based at $S_{(a_1,\dots, a_{n-1})}$:
\[
L_{(a_1,\dots, a_{n-1})}
=P(\Hull(S_{(a_1,\dots, a_{n-1})}), v, n, |\nu_n|)\cap\mathbb{Z}^n.
\]
This is a finite subset of $\mathbb{Z}^n$ whose $e_n$-height is at most
$|\nu_n|$. Since $\Hull(S_{(a_1,\dots, a_{n-1})})$ is a fixed
$(n-1)$-dimensional rectangle in $\mathbb{R}^n$ contained in the hyperplane
$\{x_n=0\}$, the size $m=|L_{(0,\dots,0)}|$ depends only on $n, d, v$.

Enumerate $L_{(0,\dots,0)}=\{x_1, x_2, \dots, x_m\}$ and define
\[
M_{(0,\dots,0)}=\{x_1,\; x_2+2dv,\; x_3+4dv,\; \dots,\;
x_m+2(m-1)dv\}.
\]
For distinct $x_k+2(k-1)dv$ and $x_{k'}+2(k'-1)dv$ with $k<k'$, consider the difference of their $e_n$-coordinates. The original points
$x_k, x_{k'}\in L_{(0,\dots,0)}$ have $e_n$-coordinates in
$[0,|\nu_n|]$, so $|x_k(n)-x_{k'}(n)|\leq |\nu_n|$. The shift
$2(k'-k)dv$ changes the $e_n$-coordinate by $2(k'-k)d\,\nu_n$, whose
absolute value is $\geq 2d|\nu_n|$ for $k'\neq k$. Hence
\[
|(x_k+2(k-1)dv)(n)-(x_{k'}+2(k'-1)dv)(n)|
\geq 2d|\nu_n|-|\nu_n|
= (2d-1)|\nu_n| \geq d,
\]
since $d\geq 1$ and $|\nu_n|\geq 1$. The Euclidean distance between two
points is at least the absolute difference in any single coordinate,
so $$\rho(x_k+2(k-1)dv,\,x_{k'}+2(k'-1)dv)\geq d.$$

For any $(a_1,\dots, a_{n-1})\in \mathbb{Z}^{n-1}$, set
\[
M_{(a_1,\dots, a_{n-1})}=M_{(0,\dots,0)}+a_1d\, e_1
+\cdots+a_{n-1}d\, e_{n-1}.
\]
The key observation is that
\[
P(\Hull(S_{(a_1,\dots, a_{n-1})}), v)\cap\mathbb{Z}^n
= L_{(a_1,\dots, a_{n-1})}+\mathbb{Z}v.
\]
Consequently, for any $x\in\mathbb{Z}^n$ that crosses
$\Hull(S_{(a_1,\dots, a_{n-1})})$ in direction $v$, there exists
$a\in\mathbb{Z}$ such that $x+av$ lies in the $\mathbb{Z}v$-orbit of some
point in $M_{(a_1,\dots, a_{n-1})}$, and a further translation along $v$
lands $x$ in $M_{(a_1,\dots, a_{n-1})}$ itself.

Now define the height parameter
\[
H(n, d, v, i)=2(m-1)d|\nu_n|+|\nu_n|.
\]
The maximum $e_n$-coordinate variation among points of $M_{(0,\dots,0)}$
is at most $2(m-1)d|\nu_n|+|\nu_n|$, so each
$M_{(a_1,\dots, a_{n-1})}$ is contained within the $e_n$-range
$\{0,\dots, H(n,d,v,i)\}$ of the strip $R$.

Finally, set
\[
M=\bigcup_{(a_1,\dots, a_{n-1})\in\mathbb{Z}^{n-1}}
M_{(a_1,\dots, a_{n-1})}.
\]
For distinct points in the same $M_{(a_1,\dots, a_{n-1})}$, the spacing
$\geq d$ follows from the construction. For points in distinct blocks
$M_{(a_1,\dots, a_{n-1})}$ and $M_{(b_1,\dots, b_{n-1})}$, the
$e_j$-coordinate difference for some $1\leq j\leq n-1$ is at least $d$
in absolute value, giving Euclidean distance $\geq d$. Property (2) holds
since every $x\in\mathbb{Z}^n$ crosses some
$\Hull(S_{(a_1,\dots, a_{n-1})})$ in direction $v$, and the construction of
$M_{(a_1,\dots, a_{n-1})}$ guarantees an appropriate translate along $v$
lands in $M$.
\end{proof}

The following lemma, which is a direct corollary of Lemma~\ref{lem:generalH}, extends the construction from axis-aligned bases to general $(n-1)$-dimensional polyhedral bases.

\begin{lemma}\label{lem:H}
Let $n, d\geq 1$ be positive integers, let $1\leq i\leq n$, and let $v=(\nu_1,\dots, \nu_n)\in\mathbb{Z}^n$ be an arbitrary vector with $\nu_i\neq 0$. Then for any $(n-1)$-dimensional polyhedron $S$ in $\mathbb{R}^n$ such that $\pi_i(S)$ is a singleton, there exists a subset $M\subseteq P(S, v, i, H(n, d, v, i))\cap\mathbb{Z}^n$ such that
\begin{enumerate}
\item[(1)] for any distinct $x,y\in M$, $\rho(x,y)\geq d$;
\item[(2)] for any $x\in \mathbb{Z}^n$ that crosses $S$ in direction $v$, there exists $a\in\mathbb{Z}$ such that $x+av\in M$.
\end{enumerate}
\end{lemma}

\subsection{Packaging and spacing}

The following one-dimensional packaging-and-spacing lemma is from \cite{GW25}. We reproduce the proof below for the convenience of the reader.

\begin{lemma}[{\cite[Lemma 3.2]{GW25}}]\label{lem:dim1packaging}
Let $m\geq 0$ be an integer and let $d, k\geq 1$ be positive integers.
Let $I$ be an interval of $\mathbb{Z}$. Let $\mathcal{J}$ be a collection
of $m$ subintervals of $I$, each of length at most $d$. If $I$ has length
at least $3d(2m+k+1)$, then there is a collection $\mathcal{K}$ of $k$ many
pairwise disjoint subintervals of $I$ such that
\begin{enumerate}
\item[(i)] each $K\in\mathcal{K}$ has length at least $d$;
\item[(ii)] for any distinct $K_1, K_2\in\mathcal{K}$,
$\rho(K_1, K_2)\geq d$;
\item[(iii)] for any $J\in\mathcal{J}$ and $K\in\mathcal{K}$,
$\rho(J, K)\geq d$.
\end{enumerate}
\end{lemma}

\begin{proof} Divide $I$ up into consecutive disjoint intervals of length $3d$, with at most one interval at the end whose length is at most $3d$. So there are at least $2m+k+1$ many such intervals. Now $\bigcup \mathcal{J}$ can intersect at most $2m$ many such intervals. Thus there are at least $k$ many such intervals with no intersections with $\bigcup\mathcal{J}$. Let $\mathcal{K}$ be the middle $\frac{1}{3}$ of these $k$ many intervals of length $3d$. Then it is easy to verify that (i)--(iii) hold for $\mathcal{K}$.
\end{proof}

We now give a multi-dimensional version adapted to polyhedral bases.

\begin{lemma}\label{lem:1directionmultiple}
Let $m\geq 0$ be an integer and let $n, d, k, H\geq 1$ be positive integers
with $1\leq i\leq n$. Let $R$ be an $n$-dimensional rectangle in
$\mathbb{Z}^n$ and let $v=(\nu_1,\dots,\nu_n)\in\mathbb{Z}^n$ with
$\nu_i\neq 0$ and $d>H+2\|v\|_2$. Let $\mathcal{Q}$ be a set of $k$ many
$(n-1)$-dimensional polyhedra in $\mathbb{R}^n$ such that each
$Q\in\mathcal{Q}$ is contained in
$\Hull(F_i^+(R))\cup\Hull(F_i^-(R))$ and has normal vector $e_i$.
Let $\mathcal{J}$ be a collection of $m$ subintervals of $\pi_i(R)$,
each of length at most $d$. Suppose $R$ has side length at least
$3d(2m+k+1)$ in direction $e_i$. Then there is an assignment
$Q\mapsto a_Q$ from $\mathcal{Q}$ to $\mathbb{Z}$ such that
\begin{enumerate}
\item[(1)] for each $Q\in\mathcal{Q}$,
$\pi_i(P(Q+a_Q v, v, i, H))\subseteq
[\pi_i(F_i^-(R)), \pi_i(F_i^+(R))]$;
\item[(2)] for distinct $Q_1, Q_2\in\mathcal{Q}$,
$\rho(P(Q_1+a_{Q_1} v, v, i, H), P(Q_2+a_{Q_2} v, v, i, H))\geq d$;
\item[(3)] for any $J\in\mathcal{J}$ and $Q\in\mathcal{Q}$,
$\rho(P(Q+a_Q v, v, i, H),
\{x\in\mathbb{Z}^n\colon \pi_i(x)\in J\})\geq d$.
\end{enumerate}
\end{lemma}

\begin{proof}
Without loss of generality assume $i=n$. Apply
Lemma~\ref{lem:dim1packaging} to obtain a collection $\mathcal{K}$ of $k$
pairwise disjoint intervals in $\pi_n(R)$ with properties (i)--(iii).
Let $Q\mapsto K_Q$ be a bijection from $\mathcal{Q}$ to $\mathcal{K}$.
Since each $Q\in\mathcal{Q}$ has normal vector $e_n$ and
$d>H+2\|v\|_2$, we can choose $a_Q\in\mathbb{Z}$ so that
$\pi_n(P(Q+a_Q v, v, n, H))\subseteq K_Q$. The spacing properties (2) and (3) then follow from the corresponding
separation of the intervals in $\mathcal{K}$.
\end{proof}

\subsection{A division lemma for polyhedra}

\begin{lemma}\label{lem:division}
Let $n, L\geq 1$ be positive integers. Let $R\subseteq\mathbb{R}^n$ be an
$n$-dimensional $\leq\!\! p$-hedron and let $S\subseteq R$ be the union of $L$
many $n$-dimensional $\leq\!\! q$-hedra. Then $R\setminus \!S$ can be written as
the union of at most $2^{qL}$ pairwise interior-disjoint $n$-dimensional
$\leq\!\! (p+qL)$-hedra.
\end{lemma}

\begin{proof}
We proceed by induction on $L$. For $L=1$, removing one $\leq\!\! q$-hedron
from a $\leq\!\! p$-hedron: the removed set is the intersection of $q$ closed
half-spaces, so its complement within $R$ is the union of at most $q$
regions, each bounded by one of the $q$ hyperplanes (with reversed
orientation) together with the original $p$ hyperplanes. Taking all
possible sign combinations for the $q$ hyperplanes yields at most $2^q$
polyhedra, each bounded by at most $p+q$ hyperplanes.

Now assume the claim for $L-1$. Write $S=S_{L-1}\cup T$ where $S_{L-1}$
is the union of $L-1$ many $\leq\!\! q$-hedra and $T$ is a single
$\leq\!\! q$-hedron. By the induction hypothesis, $R\setminus S_{L-1}$ is the
union of at most $2^{q(L-1)}$ pairwise interior-disjoint
$\leq\!\! (p+q(L-1))$-hedra. Removing $T$ from each of these splits each
component into at most $2^q$ pieces, each bounded by at most
$(p+q(L-1))+q=p+qL$ hyperplanes. The total count is at most
$2^{q(L-1)}\cdot 2^q=2^{qL}$.
\end{proof}

\subsection{Finitely many normals}

The following lemma ensures that the class of normal vectors arising in
our construction remains finite, which is essential for controlling the
combinatorial complexity of the polyhedra we need to handle.

\begin{lemma}\label{lem:finite_normals}
Let $n\geq 2$ and let
$\mathcal{N}=\{e_1,\dots,e_n, v_1,\dots,v_K\}\subseteq\mathbb{Z}^n$ and
$W\subseteq\mathbb{Z}^n$ be finite sets of integer vectors. Consider all hyperplanes in $\mathbb{R}^n$ whose normal vectors
belong to $W$. Let $L$ be the intersection of two such hyperplanes with
$\dim L=n-2$. If $\widetilde{H}$ is a hyperplane obtained by translating
$L$ parallel to some vector $v\in\mathcal{N}$, then the primitive integer
normal vector of $\widetilde{H}$ is unique up to sign and belongs to a
finite set depending only on $\mathcal{N}$ and $W$.
\end{lemma}

\begin{proof}
Let $H_1, H_2$ be hyperplanes with normals $w_1, w_2\in W$, and let
$L=H_1\cap H_2$. Let $u$ be a normal vector of $\widetilde{H}$. Let $V$
be the $(n-2)$-dimensional subspace of direction vectors parallel to $L$.

Since $\widetilde{H}$ contains a translate of $L$ and is parallel to $v$,
we have $u\perp V$ and $u\perp v$. By definition, $w_1\perp V$ and
$w_2\perp V$. Since $\dim V^\perp=2$, we have
$V^\perp=\operatorname{span}\{w_1,w_2\}$. Thus
\[
u\in \operatorname{span}\{w_1,w_2\}\cap v^\perp.
\]

We claim $\operatorname{span}\{w_1,w_2\}\not\subseteq v^\perp$. If it were,
then $v\in V$, meaning $v$ is parallel to $L$, so translating $L$ along
$v$ would not create an hyperplane, contradiction.
Hence by the dimension formula,
\[
\dim(\operatorname{span}\{w_1,w_2\}\cap v^\perp)=2+(n-1)-n=1.
\]
Thus the intersection is one-dimensional, and $u$ is uniquely determined
up to scaling. Since $w_1, w_2, v$ are integer vectors, $u$ can be chosen
as an integer vector, and its primitive integer normal is unique up to sign.

Since $\mathcal{N}$ and $W$ are finite, there are only finitely many
choices for the triple $(w_1, w_2, v)$, so the set of possible primitive
normal vectors $u$ is finite and depends only on $\mathcal{N}$ and $W$.
\end{proof}

\subsection{The main technical lemma}

The following is the central technical lemma of the paper. Given an
$n$-dimensional polyhedron $R'$, we construct a family of
$(n-1)$-dimensional cross-sectional polyhedra that serve as bases for
marker packages.

\begin{lemma}\label{lem:mainpoly}
Let $n, d, m, H\geq 1$ be positive integers, let $1\leq i\leq n$, let
$v=(\nu_1,\dots,\nu_n)\in \mathbb{Z}^n$ with $\nu_i\neq0$, let
$v_1,\dots, v_m\in \Z^n$, and let $0<t_0<1$ satisfy
\[
t_0\leq\min\left\{\frac{\proj_{v_j} v}{v_j}>0\colon 1\leq j\leq m\right\},
\]
where $\frac{\proj_u w}{u}$ denotes the unique scalar $\lambda$ such that the
orthogonal projection of $w$ onto the line spanned by $u$ equals $\lambda u$.
If $d t_0>40m\|v\|_\infty^2 H$, then for any $n$-dimensional
$\leq\! m$-hedron $R'=\operatorname{PH}((P_1,v_1),\dots,(P_m,v_m))$ in
$\mathbb{R}^n$, there exist
\[
K\leq m\left\lceil \frac{8m|\pi_i(R')|\cdot\|v\|_\infty}{d}+1\right\rceil
\]
and a finite sequence of $(n-1)$-dimensional polyhedra
$\{B_k\}_{1\leq k\leq K}$ in $\mathbb{R}^n$ such that
\begin{enumerate}
\item[(1)] each $B_k$ has $e_i$ as its normal vector;
\item[(2)] for any $1\leq k\neq k'\leq K$,
$\rho(P(B_k, v, i, H),P(B_{k'}, v, i, H))\geq \dfrac{d t_0}{10m\|v\|_\infty}$;
\item[(3)] for any $1\leq k\leq K$,
$\dfrac{d}{3}\geq\rho(P(B_k, v, i, H),R')\geq \dfrac{d t_0}{10m\|v\|_\infty}$;
\item[(4)] for any $1\leq i'\leq n$,
$|\pi_{i'}(B_k)|\leq |\pi_{i'}(R')|+\dfrac{d}{8m}$;
\item[(5)] for any $x\in\mathbb{R}^n$, if $x$ crosses $R'$ in direction $v$,
then there exists $1\leq k\leq K$ such that $x$ crosses $B_k$ in direction
$v$.
\end{enumerate}
\end{lemma}

\begin{proof}
Without loss of generality assume $i=n$.  Among the $m$ face normals of $R'$,
some may not have a positive projection of $v$ onto them, and some of the
corresponding hyperplanes may be redundant (they do not touch $R'$).  Assume 
\begin{align*}
1\leq j\leq m' &\iff \exists q>0\Bigl(\proj_{v_j} v=q v_j\Bigr),\\
1\leq j\leq m'' &\iff \bigl(j\leq m'\;\wedge\; P_j\cap R'\neq\varnothing\bigr),
\end{align*}
so $m''\leq m'\leq m$.  After renumbering we may assume the first $m''$
indices are exactly those with genuine faces whose normals admit a positive
projection of $v$; faces with index $j>m''$ play no further role. Our proof takes several steps.

\medskip\noindent
\textbf{Step 1.  Displacing the faces.}
For each $1\leq j\leq m''$, let $F_j=P_j\cap R'$ be the corresponding face
of $R'$.  Translate $F_j$ outward along $v$ by a distance proportional to
its index:
\[
P_j' = F_j+\frac{(2j-1)d}{8m\|v\|_\infty}\,v .
\]
The translated faces $\{P_j'\}$ satisfy the following three properties.
\begin{itemize}
\item[(i)] For $j\neq j'$, the displacement parameters differ by at least
$\frac{2d}{8m\|v\|_\infty}$.  The component of $v$ orthogonal to the $j$-th
face is measured by $\proj_{v_j}v$; after normalising, the definition of
$t_0$ guarantees this orthogonal component is at least $t_0$.  Hence
$\rho(P_j',P_{j'}') \geq \frac{d t_0}{4m\|v\|_\infty}$.
\item[(ii)] $\frac{(2m-1)d}{8m}\geq\rho(R', P_j') \geq \frac{d t_0}{8m\|v\|_\infty}$.
The upper bound holds because the displacement length is at most
$\frac{(2m''-1)d}{8m}\leq\frac{(2m-1)d}{8m}$; the lower bound follows from
the same projection estimate as in (i), applied to the smallest displacement
$j=1$.
\item[(iii)] Translating a face along $v$ preserves the set of lines parallel
to $v$ that intersect it.  Hence any $x$ that crosses $F_j$ in direction $v$
also crosses $P_j'$ in direction $v$.
\end{itemize}

\medskip\noindent
\textbf{Step 2.  Forming $n$-dimensional packages and slicing.}
For each $1\leq j\leq m''$, thicken $P_j'$ in direction $v$ to
$e_n$-height $\frac{d}{8m\|v\|_\infty}|\nu_n|$:
\[
Q_j=P\Bigl(P_j',\, v,\, n,\; \frac{d}{8m\|v\|_\infty}|\nu_n|\Bigr)
=\bigcup_{0\leq t\leq \frac{d}{8m\|v\|_\infty}}(P_j'+t v).
\]
For any point $x\in Q_j$, write $x=p+t_0 v$ with $p\in P_j'$.  The line
through $x$ parallel to $v$ is $\{p+sv\colon s\in\mathbb{R}\}$, and its
intersection with $Q_j$ is exactly $\{p+sv\colon
0\leq s\leq\frac{d}{8m\|v\|_\infty}\}$, whose $e_n$-projection has length
$\frac{d}{8m\|v\|_\infty}|\nu_n|$.  Thus Lemma~\ref{lem:layer} applies to
each $Q_j$ with direction $v$ and parameter $\frac{d}{8m\|v\|_\infty}|\nu_n|$,
yielding a finite collection $\mathcal{Q}_j$ of $(n-1)$-dimensional polyhedra.
These collections satisfy:
\begin{itemize}
\item[(iv)] every $Q\in\bigcup_j\mathcal{Q}_j$ has $e_n$ as normal vector;
\item[(v)] for distinct $Q,Q'\in\bigcup_j\mathcal{Q}_j$, if they belong to
the same $\mathcal{Q}_j$ then Lemma~\ref{lem:layer}(2) gives
$\rho(Q,Q')\geq\frac{d}{8m\|v\|_\infty}|\nu_n|$; if they belong to different
$\mathcal{Q}_j$, $\mathcal{Q}_{j'}$, the separation follows from (i) together
with the fact that $Q_j$ and $Q_{j'}$ are obtained by translating their
respective bases by at most $\frac{d}{8m\|v\|_\infty}v$.  In either case
$\rho(Q,Q')\geq\frac{d t_0}{8m\|v\|_\infty}$;
\item[(vi)] Lemma~\ref{lem:layer}(2) and the displacement bounds (ii) yield
$\frac{d}{4}\geq\rho(R',Q)\geq\frac{d t_0}{8m\|v\|_\infty}$;
\item[(vii)] the $e_{i'}$-projection of $Q_j$
is controlled by $\pi_{i'}(R')$ and the thickness,
hence $|\pi_{i'}(Q)|\leq |\pi_{i'}(R')|+\frac{d|\nu_i|}{8m\|v\|_{\infty}}\leq |\pi_{i'}(R')|+\frac{d}{8m}$ for every
$Q\in\bigcup_j\mathcal{Q}_j$;
\item[(viii)] if $x$ crosses $R'$ in direction $v$, then $x$ crosses some
face $F_j$ ($1\leq j\leq m''$) in direction $v$.  By (iii), $x$ crosses
$P_j'$, hence crosses $Q_j$.
Lemma~\ref{lem:layer}(3) then supplies some $Q\in\mathcal{Q}_j$ crossed by
$x$ in direction $v$;
\item[(ix)] Lemma~\ref{lem:layer} gives
$|\mathcal{Q}_j|\leq\bigl\lceil |\pi_n(Q_j)|/
\bigl(\frac{d}{8m\|v\|_\infty}|\nu_n|\bigr)\bigr\rceil$.  Since
$|\pi_{i'}(Q)|\leq |\pi_{i'}(R')|+\frac{d|\nu_n|}{8m\|v\|_{\infty}}$, we obtain
$|\mathcal{Q}_j|\leq
\bigl\lceil\frac{8m|\pi_n(R')|\cdot\|v\|_\infty}{d}+1\bigr\rceil$.
\end{itemize}

\medskip\noindent
\textbf{Step 3.  Assembling the $B_k$ and verifying (1)--(5).}
Enumerate $\bigcup_{j=1}^{m''}\mathcal{Q}_j=\{B_1,\dots,B_K\}$.  Then
$K=\sum_{j=1}^{m''}|\mathcal{Q}_j|
\leq m\bigl\lceil\frac{8m|\pi_n(R')|\cdot\|v\|_\infty}{d}+1\bigr\rceil$.

Property (1) is immediate from (iv).  Property (4) and (5) are exactly (vii) and (viii).

For property (2), let $k\neq k'$.  The distance between the bases $B_k$ and
$B_{k'}$ is at least $\frac{d t_0}{8m\|v\|_\infty}$ by (v).  Extending each
base to a package $P(B_k,v,n,H)$ adds at most $H\|v\|_\infty$ to the distance
in the worst case (when the two packages are shifted toward each other).  The
hypothesis $d t_0>40m\|v\|_\infty^2 H$ ensures this perturbation is less than
$\frac{d t_0}{40m\|v\|_\infty}$, so after collecting constants the spacing
remains at least $\frac{d t_0}{10m\|v\|_\infty}$.

For property (3), the upper bound holds because each $B_k$ originates from a
face $F_j\subseteq R'$, is displaced by at most
$\frac{(2m-1)d}{8m\|v\|_\infty}$ along $v$, and the package adds height
$H$; using $d t_0>40m\|v\|_\infty^2 H$ (which implies
$H<\frac{d}{40m\|v\|_\infty}$) and $t_0<1$, the total distance from $R'$ is
bounded by $\frac{d}{4}+\frac{d}{8m}+\frac{d}{40m}<\frac{d}{3}$.  The lower
bound follows from (vi): extending the base $B_k$ to the full package can
bring points closer to $R'$ by at most $H\|v\|_\infty$, and the hypothesis
prevents this from breaching $\frac{d t_0}{10m\|v\|_\infty}$.

This completes the proof.
\end{proof}

\subsection{An extension lemma}

The following lemma bounds the size of $\delta$-extensions of polyhedra
whose face normals belong to a fixed finite set.

Let $S$ be a finite subset of $\Z^n$, define 
$$\operatorname{cone}(S) = \left\{ \sum_{v \in S} \alpha_v v \;:\; \alpha_v \ge 0 \right\}.$$

\begin{lemma}\label{lem:extension}
Let $m, n$ be positive integers and let $A\subseteq\mathbb{Z}^n$ be a
finite set of non-zero vectors. Then there exists a positive integer
$\lambda(A,m)$ such that for any $\delta>0$ and any $n$-dimensional $\leq\!\! m$-hedron
$H=\operatorname{PH}((P_1,v_1),\dots,(P_m,v_m))$ with $v_i\in A$, the
$\delta$-extension of $H$ is contained in the
$\lambda(A,m)\delta$-neighborhood of $H$.
\end{lemma}

\begin{proof}
By homogeneity, it suffices to prove the statement for $\delta=1$; the
general case follows by scaling. Let
$H=\operatorname{PH}((P_1,v_1),\dots,(P_m,v_m))$ with $v_i\in A$, and
let $H_1$ be its $1$-extension. For each plane $P_i$, fix $p_i\in P_i$.

Let $x\in H_1\backslash H$. Since $H$ is closed and convex, there is a unique closest
point $y\in H$ to $x$. Set $u=x-y$. The vector
$u$ lies in the normal cone of $H$ at $y$:
$u=\sum_{i\in I(y)}\alpha_i v_i$ with $\alpha_i\geq 0$, where
$I(y)=\{i\colon y\in P_i\cap H\}$.

Because $x$ belongs to the $1$-extension of $H$, for each $i\in I(y)$
we have
\[
(x-p_i)\cdot v_i\leq \|v_i\|_2.
\]
Since $y\in P_i$, we have $(y-p_i)\cdot v_i=0$, and subtraction gives
\[
u\cdot v_i = (x-y)\cdot v_i \le \norm{v_i}_2,
\]
or equivalently,
\[
u\cdot \frac{v_i}{\norm{v_i}_2} \le 1.
\]

Let $S=\{v_i\colon i\in I(y)\}\subseteq A$. Then $|S|\leq m$ and
$u\in K_S$, where
\[
K_S=\Bigl\{w\in\cone(S)\;:\;
w\cdot\frac{v}{\|v\|_2}\leq 1\text{ for all }v\in S\Bigr\}.
\]
We claim $K_S$ is bounded. If not, there is a sequence
$\{u_k\}\subseteq K_S$ with $\|u_k\|_2\to\infty$. Define
$w_k=u_k/\|u_k\|_2$. By compactness of the unit sphere in $\mathbb{R}^n$,
a subsequence converges to some $w$ with $\|w\|_2=1$. Since $\cone(S)$
is closed, $w\in\cone(S)$, so $w=\sum_{v\in S}\beta_v v$ with
$\beta_v\geq 0$. The defining inequalities give
$w_k\cdot\frac{v}{\|v\|_2}\leq 1/\|u_k\|_2$ for each $v\in S$, and
taking $k\to\infty$ yields $w\cdot v\leq 0$ for all $v\in S$. Then
\[
\|w\|_2^2=w\cdot w=\sum_{v\in S}\beta_v(w\cdot v)\leq 0,
\]
contradicting $\|w\|_2=1$. Therefore $K_S$ is bounded. Being closed and bounded in $\R^n$, it is compact, and we may define $$\lambda_S = \max_{u\in K_S}\norm{u}_2 < \infty.$$

Since $A$ is finite, there are only finitely many subsets
$S\subseteq A$ with $|S|\leq m$. Define
\[
\lambda(A,m)=\max_{\substack{S\subseteq A\\ 1\leq |S|\leq m}}
\lceil\lambda_S\rceil,
\]
which is a well-defined positive integer. Now for $x\in H_1$, with
$S=\{v_i\colon i\in I(y)\}$, we have
$\|x-y\|_2=\|u\|_2\leq\lambda_S\leq\lambda(A,m)$. Thus
$H_1\subseteq H+B(0,\lambda(A,m))$, where $B(0,r)$ is the Euclidean
ball of radius $r$.
\end{proof}

\section{Proof of the main theorem}\label{sec:proof}

In this section we prove Theorem~\ref{thm:main}. The proof adapts the
inductive construction of packages from \cite{GW25}, with polyhedral
packages replacing rectangular ones. For each generating vector, we
construct a family of packages that are pairwise well-separated and such
that every orbit point can reach a marker point by moving along that
generating vector. The construction proceeds in levels (corresponding
to coordinate directions) and rounds (corresponding to generators with
non-zero entry in that coordinate).

\subsection{Fixing parameters}

Fix $n, d_0\geq 1$ and a finite generating set $S\subseteq\mathbb{Z}^n$.
Without loss of generality assume $0\notin S$ and $S$ is symmetric
(i.e., $-v\in S$ whenever $v\in S$). For $1\leq i\leq n$, define
\[
S_i=\{v\in S\colon \pi_i(v)\neq 0\},
\]
and enumerate $S_i=\{v_1^i, v_2^i, \dots, v_{m_i}^i\}$. Without loss of
generality, assume all $m_i>0$ (if some $S_i=\varnothing$, the corresponding
level is vacuous and can be omitted).

Define
\[
t=\max\{\|v\|_{\infty}\colon v\in S\},\qquad
N_0=(4t+5)^n.
\]
Here $N_0$ bounds the number of marker regions of side length $D+1$
that can intersect a $\rho_\infty$-ball of radius $(2t+1)(D+1)$; see the
discussion of square neighborhoods below.

\medskip\noindent
\textbf{Finite normals hierarchy.}
Let $W_0=\{e_1,\dots,e_n\}\cup S$. For $k\geq 1$, define $W_k$ inductively
as the set of all primitive integer normal vectors of hyperplanes obtained
by:
\begin{itemize}
\item taking the intersection of two hyperplanes whose normals belong to
$W_{k-1}$, yielding an $(n-2)$-dimensional affine subspace $L$;
\item translating $L$ parallel to some vector $v\in S$ to obtain a hyperplane.
\end{itemize}
By Lemma~\ref{lem:finite_normals}, each $W_k$ is a finite set depending
only on $S$. Let $Q=m_1+\cdots+m_n$ be the total number of rounds. Set
\[
A=W_Q,\qquad a=|A|.
\]
Applying Lemma~\ref{lem:extension} with this $A$ and $m=a$, we obtain
the constant $\lambda=\lambda(A,a)$.

\medskip\noindent
\textbf{Package height.}
For each $1\leq i\leq n$ and $v\in S_i$,
Lemma~\ref{lem:generalH} provides $H_i(v)=H(n,d_0,v,i)$. Set
\[
H=\max\{H_i(v)+2d_0\|v\|_2\colon 1\leq i\leq n, v\in S_i\}.
\]
Let $t_0$ be a positive real number with $0<t_0<1$ such that
\[
t_0\leq \min\Bigl\{\frac{\proj_u v}{u}>0\colon
v\in S,\; u\in A\Bigr\},
\]
where $\proj_u v$ denotes the projection of $v$ onto $u$. Such $t_0$
exists because $S$ and $A$ are finite sets.

\medskip\noindent
\textbf{Inductive parameters.}
We define sequences of parameters $l_k, b_k, c_k$ for $k=1,\dots, Q$ by
induction. Their roles are:
\begin{itemize}
\item $l_k$: an upper bound for $|\pi_{i'}(P)|/(D+1)$ for all
$1\leq i'\leq n$ and all polyhedra $P$ constructed up to step $k$;
\item $b_k$: an upper bound for the number of polyhedra per marker
region after step $k$;
\item $c_k$: a spacing parameter; after step $k$, any two distinct
polyhedra have Euclidean distance at least $D/c_k$.
\end{itemize}

Set $l_1=1$, $b_1=2$, and let $c_1$ be an integer such that
$c_1>3(2N_0b_1+7)$.

Now assume $l_k, b_k, c_k$ have been defined for some $k<Q$. Let
$c_{k+1}'$ be an integer satisfying
\begin{equation}\label{eq:cprime}
c_{k+1}'>3\lambda c_k.
\end{equation}
Define
\begin{equation}\label{eq:l}
l_{k+1}' = l_k + \frac{2\lambda}{c_{k+1}'},\qquad
l_{k+1} = l_{k+1}' + \frac{\lambda}{8a c_{k+1}'} + H t,
\end{equation}
\begin{equation}\label{eq:b}
b_{k+1}' = a b_k
\Bigl\lceil \frac{8a c_{k+1}' l_{k+1}' t}{\lambda}+2 \Bigr\rceil
+ b_k,\qquad
b_{k+1} = b_{k+1}' + 2^{1+N_0 a b_{k+1}'},
\end{equation}
and let $c_{k+1}$ be an integer satisfying
\begin{equation}\label{eq:c}
c_{k+1} > 3(2N_0b_{k+1}+b_{k+1}+5) + c_{k+1}'
+ \frac{10a c_{k+1}' t}{\lambda t_0}
+ \frac{3c_{k+1}'}{7\lambda}.
\end{equation}

\medskip\noindent
\textbf{Final parameters.}
Choose an integer $D$ large enough so that
\begin{equation}\label{eq:D}
D > 8a\, c_Q\, l_Q\, d_0\, t,
\end{equation}
and choose $\Delta$ such that
\begin{equation}\label{eq:Delta}
\Delta > \sqrt{n}\,(6t+6)(D+1).
\end{equation}

\subsection{Setting up the clopen framework}

Apply Lemma~\ref{lem:markerregions} with parameter $D$ to obtain a clopen
equivalence relation $E_D$ on $F(2^{\mathbb{Z}^n})$ whose equivalence
classes are $n$-dimensional rectangles with side
lengths either $D$ or $D+1$. For each marker region $R$, let $x_R$ be its
``least'' corner (index $1$ in the canonical enumeration of extreme
points). Set
\[
X=\{x_R\colon R\text{ is a marker region}\},
\]
which is a clopen subset of $F(2^{\mathbb{Z}^n})$.

Apply Lemma~\ref{lem:basicmarker} with parameter $d=\Delta$ to obtain a
clopen set $M_\Delta\subseteq F(2^{\mathbb{Z}^n})$ such that distinct
points of $M_\Delta$ in the same orbit are at Euclidean distance
$\geq\Delta$, and every orbit point is within Euclidean distance
$<\sqrt{n}\,\Delta$ of some point of $M_\Delta$.

Enumerate $\{g\in\mathbb{Z}^n\colon \|g\|_2\leq \sqrt{n}\,\Delta\}$ as
$g_1,\dots,g_K$. Partition $X$ as $\{X_1,\dots,X_K\}$ where
\[
X_1=g_1\cdot M_\Delta\cap X,\qquad
X_k=g_k\cdot M_\Delta\cap X\setminus
\bigcup_{1\leq j<k}X_j\;\; \mbox{ for $1<k\leq K$.}
\]

For any $1\leq k\leq K$ and distinct marker regions $R,T$ with
$x_R,x_T\in X_k$, we have
\[
\rho(x_R,x_T)\geq\Delta>\sqrt{n}\,(6t+6)(D+1),
\]
which by the equivalence of norms implies
$\rho_\infty(x_R,x_T)>(6t+6)(D+1)$. Consequently, the square
$(3t+2)(D+1)$-neighborhoods (with respect to $\rho_\infty$) of $R$ and
$T$ are disjoint. In particular, $R$ and $T$ are far apart, and
constructions performed in $R$ do not interfere with constructions in
$T$ during the same step.

\subsection{Inductive construction of packages}

We construct families of $n$-dimensional polyhedra in $F(2^{\mathbb{Z}^n})$
organized by levels and rounds. For level $i$ ($1\leq i\leq n$) and round
$j$ ($1\leq j\leq m_i$), we handle the generator $v_j^i\in S_i$.

For each marker region $R$, we will construct a collection
$\mathcal{P}_{i,j}^R$ of $n$-dimensional polyhedra (viewed as subsets of
the orbit $[x_R]$ identified with $\mathbb{Z}^n$). Define
\[
\mathcal{P}_{i,j,k}=\bigcup_{x_R\in X_k}\mathcal{P}_{i,j}^R,\qquad
\mathcal{P}_{i,j}=\bigcup_{k=1}^K\mathcal{P}_{i,j,k}.
\]

The construction proceeds in $K$ steps within each round: in step $k$, we
work on all marker regions $R$ with $x_R\in X_k$ simultaneously. Since
these regions are pairwise far apart in the $\rho_\infty$-metric (hence
also in the Euclidean metric), the constructions are independent.

Set
$\mathcal{Q}_{i,j}^R=
\bigcup_{(i',j')\leq_{\text{lex}}(i,j)}\mathcal{P}_{i',j'}^R$.
Let $k(i,j)=m_1+\cdots+m_{i-1}+j$ be the global step index.

The collections will satisfy the following inductive hypotheses for every
marker region $R$:

\begin{enumerate}
\item[(I1)] For each $P\in\mathcal{P}_{i,j}$, there is an
$(n-1)$-dimensional polyhedron $B_P\subseteq\mathbb{R}^n$ with normal
vector $e_i$ and integer $\pi_i$-coordinate, such that under the natural
identification of the orbit with $\mathbb{Z}^n$, $P$ corresponds to
$P(B_P, v_j^i, i, H)\cap\mathbb{Z}^n$.

\item[(I2)] (Covering property) Under the identification of the orbit
with $\mathbb{Z}^n$ that sends $x_R$ to $\overline{0}$ and preserves
relative positions, for any $x\in\mathbb{Z}^n$ that crosses
$\Hull(F_i^+(R))$ or $\Hull(F_i^-(R))$ in direction $v_j^i$, there
exists a polyhedron $P$ belonging to $\mathcal{P}_{i,j}$ (possibly
constructed for a nearby marker region) such that $x$ crosses $B_P$ in
direction $v_j^i$.

\item[(I3)] For any $1\leq i'\leq n$ and any $P\in\mathcal{P}_{i,j}$,
the $e_{i'}$-coordinate extent satisfies
$|\pi_{i'}(P)|\leq l_{k(i,j)}(D+1)$.

\item[(I4)] $|\mathcal{Q}_{i,j}^R|\leq b_{k(i,j)}$.

\item[(I5)] For any two distinct polyhedra $P, Q$ belonging to the union
$\bigcup_{R}\mathcal{Q}_{i,j}^R$, the Euclidean
distance satisfies $\rho(P,Q)\geq D/c_{k(i,j)}$.

\item[(I6)] For any $P\in\mathcal{P}_{i,j}^R$, all points of $P$ are
within the $\rho_\infty$-distance $t(D+1)$ of the marker region $R$.
\end{enumerate}

\subsubsection{First level, first round: constructing $\mathcal{P}_{1,1}$}

Fix a step index $k$ ($1\leq k\leq K$) and a marker region $R$ with
$x_R\in X_k$.

\medskip
\noindent\textbf{Copying to a ``scratch'' $\mathbb{Z}^n$.}
Identify the orbit of $x_R$ with $\mathbb{Z}^n$ via the unique group
element sending $x_R$ to $\overline{0}$. Under this identification,
copy $R$ and all marker regions $R'$ that intersect the square
$(2t+1)(D+1)$-neighborhood of $R$, and
all previously constructed polyhedra
$P\in\mathcal{P}_{1,1}^{R'}$ to
$\mathbb{Z}^n$, preserving their relative positions. By a standard volume
argument and the definition of $N_0$, there are at most $N_0$ such
regions $R'$.

\medskip
\noindent\textbf{Applying the packaging and spacing lemma.}
Consider the two sets $\Hull(F_1^+(R))$ and $\Hull(F_1^-(R))$ as
$(n-1)$-dimensional polyhedra in $\mathbb{R}^n$; both have normal vector
$e_1$. Apply Lemma~\ref{lem:1directionmultiple} to $R$ with parameters
\[
v=v_1^1,\quad d=\frac{D}{c_1},\quad H\text{ as fixed above},\quad
\mathcal{P}=\{\Hull(F_1^+(R)),\Hull(F_1^-(R))\} \text{ and}
\]
\[
\mathcal{J}=\pi_1\left (\bigcup_{R'}
\mathcal{P}_{1,1}^{R'}\right )\cup\{0, D\}.
\]
The side length condition of Lemma~\ref{lem:1directionmultiple} is
satisfied because
\[
D > 3\frac{D}{c_1}(2(N_0b_1+2)+2+1).
\]
 We obtain integers $a,b\in\mathbb{Z}$ and two
new packages in $\mathbb{R}^n$:
\[
P_+'=P(\Hull(F_1^+(R))+av_1^1,\; v_1^1,\; 1,\; H),\qquad
P_-'=P(\Hull(F_1^-(R))+bv_1^1,\; v_1^1,\; 1,\; H),
\]
such that $\pi_1(P_+')\cup\pi_1(P_-')\subseteq\pi_1(R)$.

Let $P_+, P_-\subseteq F(2^{\mathbb{Z}^n})$ be the subsets corresponding
to $P_+'\cap\mathbb{Z}^n$ and $P_-'\cap\mathbb{Z}^n$ under the inverse
identification. Define
\[
\mathcal{P}_{1,1}^R=\{P_+, P_-\},
\]
with base assignments
$B_{P_+}=\Hull(F_1^+(R))+av_1^1$ and
$B_{P_-}=\Hull(F_1^-(R))+bv_1^1$.

We verify the inductive hypotheses. (I1) holds by construction. (I2) holds
because any point crossing $\Hull(F_1^+(R))$ or $\Hull(F_1^-(R))$ in
direction $v_1^1$ crosses the displaced base $B_{P_+}$ or $B_{P_-}$
respectively. (I3) follows from
$|\pi_{i'}(P_\pm)|\leq D+1 = l_1(D+1)$. (I4) holds with
$|\mathcal{Q}_{1,1}^R|=2=b_1$. (I5) holds by the spacing guaranteed by
Lemma~\ref{lem:1directionmultiple}. (I6) holds because $\pi_1(P_+')\cup\pi_1(P_-')\subseteq\pi_1(R)$. This finishes the definition of $\mathcal{P}_{1,1}$.

\subsubsection{Inner induction: constructing $\mathcal{P}_{1,j+1}$ from
$\mathcal{P}_{1,1},\dots,\mathcal{P}_{1,j}$}

Now assume we have constructed
$\mathcal{P}_{1,1},\dots,\mathcal{P}_{1,j}$ for some $1\leq j<m_1$
satisfying (I1)--(I6). We construct $\mathcal{P}_{1,j+1}$.

Fix step $s$ and a marker region $R$ with $x_R\in X_s$. As before,
identify the orbit containing $R$ with $\mathbb{Z}^n$ (mapping $x_R$
to $\overline{0}$) and copy $R$, all nearby marker regions $R'$, and
all previously constructed polyhedra
$P\in\mathcal{P}_{1,j'}^{R'}$ ($1\leq j'\leq j+1$) to $\mathbb{Z}^n$,
preserving relative positions.

Apply Lemma~\ref{lem:1directionmultiple} to $R$ with
\[
v=v_{j+1}^1,\quad d=\frac{D}{c_{j+1}},\quad
\mathcal{P}=\{\Hull(F_1^+(R)),\Hull(F_1^-(R))\} \text{ and}
\]
\[
\mathcal{J}=\pi_1\left (\bigcup_{R',\,1\leq j'\leq j+1}
\mathcal{P}_{1,j'}^{R'}\right )\cup\{0, D\}.
\]
The set $\mathcal{J}$ collects the $e_1$-projections of all previously
constructed packages near $R$, together with the two boundary points.
By (I4) and (I6), the number of previously constructed packages whose
$e_1$-projection can intersect $\pi_1(R)$ is at most $N_0 b_j$, so
$|\mathcal{J}|\leq N_0 (b_j+2)$. The side length condition
$D>3d(2|\mathcal{J}|+2+1)$ follows from (\ref{eq:D}) and (\ref{eq:c}).

We obtain integers $a,b$ and define $P_+, P_-$ analogously to the base
case, with bases $B_{P_+}=\Hull(F_1^+(R))+av_{j+1}^1$ and
$B_{P_-}=\Hull(F_1^-(R))+bv_{j+1}^1$. Set
$\mathcal{P}_{1,j+1}^R=\{P_+, P_-\}$.

The verification of (I1)--(I6) is as in the base case. This completes the definition of the first level.

\subsubsection{Level change: constructing $\mathcal{P}_{i+1,1}$ from
$\mathcal{P}_{1,1},\dots,\mathcal{P}_{i,m_i}$}

Assume we have constructed all packages through level $i$ ($1\leq i<n$)
and round $m_i$. We now construct $\mathcal{P}_{i+1,1}$. Let
$k=m_1+\cdots+m_i$ be the global step index before this construction.

Fix step $s$ and a marker region $R$ with $x_R\in X_s$. Identify the
orbit with $\mathbb{Z}^n$ as before and copy all relevant data ($R$,
nearby marker regions, and all previously constructed polyhedra).

\medskip\noindent
\textbf{Part A: Packages from existing polyhedra.}
For each $P\in\mathcal{Q}_{i,m_i}^R$, let $P'\subseteq\mathbb{R}^n$ be
the polyhedron corresponding to $P$ under the identification (so
$P=P'\cap\mathbb{Z}^n$). Let $P^*$ be the $(D/c_{k+1}')$-extension of
$P'$. By
Lemma~\ref{lem:extension}, $P^*$ is contained in the Euclidean
$(\lambda D/c_{k+1}')$-neighborhood of $P'$.

Apply Lemma~\ref{lem:mainpoly} to each $P^*$ with parameters
\[
n,\quad d=\frac{\lambda D}{c_{k+1}'},\quad m=a,\quad H,\quad
i=i+1,\quad v=v_1^{i+1}.
\]
The hypotheses of Lemma~\ref{lem:mainpoly} are satisfied because $P^*$
is an $n$-dimensional $\leq\!\! a$-hedron (all normals of faces of $P^*$ belong to $A$ by construction). We
obtain, for each $P\in\mathcal{Q}_{i,m_i}^R$, a finite collection
\[
\mathcal{B}_{i,P}^R=\{B_\ell\}_{\ell}
\]
of $(n-1)$-dimensional polyhedra satisfying conditions (1)--(5) of
Lemma~\ref{lem:mainpoly}. Define
\[
\mathcal{B}_{i}^R=\bigcup_{P\in\mathcal{Q}_{i,m_i}^R}
\mathcal{B}_{i,P}^R.
\]

For each $B\in\mathcal{B}_i^R$, let
$P_B'=P(B, v_1^{i+1}, i+1, H)$ and let $P_B\subseteq F(2^{\mathbb{Z}^n})$
be the corresponding set. Set
$\mathcal{C}_{i}^R=\{P_B\colon B\in\mathcal{B}_i^R\}$.

\medskip\noindent
\textbf{Part B: Packages from the faces of $R$.}
Consider $\Hull(F_{i+1}^+(R))$ and $\Hull(F_{i+1}^-(R))$ in
$\mathbb{R}^n$. Remove those points $x$ in these sets for which there
exists $a\in\mathbb{R}$ with
$\pi_{i+1}(x+av_1^{i+1})\subseteq\pi_{i+1}(R)$ and such that
$x+av_1^{i+1}$ belongs to some $B\in\mathcal{B}_i^{R'}$ (where $R'$ is a
nearby marker region whose data was copied). In other words, we remove
the ``shadow'' of bases that have already been covered by Part~A.

By Lemma~\ref{lem:division}, the remainder of
$\Hull(F_{i+1}^+(R))\cup\Hull(F_{i+1}^-(R))$ can be written as the
union of at most $2^{1+N_0a b_{k+1}'}$ pairwise interior-disjoint
$n$-dimensional polyhedra (the $1$ in the exponent accounts for the
two faces). Denote this collection by $\mathcal{O}_i^R$.

Apply Lemma~\ref{lem:1directionmultiple} with
\[
v=v_1^{i+1},\quad d=\frac{D}{c_{k+1}},\quad
\mathcal{P}=\mathcal{O}_i^R,\quad
\mathcal{J}=\pi_{i+1}\left (\bigcup_{R'}(\mathcal{B}_i^{R'}\cup \mathcal{U}_i^{R'})\right )
\cup\{0, D\},
\]
where $\mathcal{U}_i^{R'}=\emptyset$ if it has not been defined, 
to obtain, for each $O\in\mathcal{O}_i^R$, an integer $a_O$ and a new
package. Let $\mathcal{U}_i^R$ be the collection of these new packages
(translated to $F(2^{\mathbb{Z}^n})$).

\medskip\noindent
\textbf{Assembling the new collection.}
Define
\[
\mathcal{P}_{i+1,1}^R=\mathcal{C}_i^R\cup\mathcal{U}_i^R,
\]
with the base assignments inherited from Part~A and Part~B respectively.

We verify the inductive hypotheses:
\begin{itemize}
\item (I1) holds by construction for both $\mathcal{C}_i^R$ and
$\mathcal{U}_i^R$.
\item (I2) holds: for the faces $F_{i+1}^+(R)$ and $F_{i+1}^-(R)$, the
construction of $\mathcal{U}_i^R$ guarantees that points crossing these
faces in direction $v_1^{i+1}$ are covered (the removed ``shadow''
corresponds exactly to points already covered by $\mathcal{C}_i^R$ via the
crossing property of Lemma~\ref{lem:mainpoly}(5)). The covering property
of Lemma~\ref{lem:mainpoly}(5) also ensures that points crossing previously
constructed polyhedra (now extended) in direction $v_1^{i+1}$ are covered
by $\mathcal{C}_i^R$.
\item (I3) follows from the definition of $l_{k+1}$ and
Lemma~\ref{lem:mainpoly}(4) together with the extension bound.
\item (I4) follows from the estimate
\[
|\mathcal{B}_i^R| \leq a b_k
\Bigl\lceil \frac{8a c_{k+1}' l_{k+1}' t}{\lambda}+2 \Bigr\rceil
\leq b_{k+1}'-b_k
\]
(by Lemma~\ref{lem:mainpoly} and (\ref{eq:b})), so
$|\mathcal{C}_i^R| \leq b_{k+1}'-b_k$, together with
$|\mathcal{O}_i^R|\leq 2^{1+N_0ab_{k+1}'}$, giving
$|\mathcal{Q}_{i+1,1}^R|\leq b_{k+1}$.
\item (I5) follows from the spacing properties of
Lemma~\ref{lem:mainpoly}(2)(3) (for packages within $\mathcal{C}_i^R$) and
Lemma~\ref{lem:1directionmultiple} (for packages in $\mathcal{U}_i^R$ and
between $\mathcal{C}_i^R$ and $\mathcal{U}_i^R$), together with the choice
of $c_{k+1}$.
\item (I6) follows because all constructions stay within the
$\rho_\infty$-distance $t(D+1)$ of $R$, by the bound on $H$ and the fact
that we only use translates along $v$ whose $e_{i+1}$-coordinate stays
within $\pi_{i+1}(R)$.
\end{itemize}

\subsubsection{Subsequent rounds at level $i+1$}

The construction of $\mathcal{P}_{i+1,j+1}$ from the previously
constructed $\mathcal{P}_{1,1},\dots,\mathcal{P}_{i+1,j}$ follows the
same two-part pattern: for each existing polyhedron, apply
Lemma~\ref{lem:mainpoly} to its extension; for the faces of $R$, remove
points already covered and use Lemma~\ref{lem:division} followed by
Lemma~\ref{lem:1directionmultiple}. The generating vector changes to
$v_{j+1}^{i+1}$ and the parameters $d, c, b, l$ advance to the next
global index. The same verifications ensure that the inductive hypotheses
(I1)--(I6) are maintained throughout.

\subsection{Constructing the marker set $M$}

After completing all $Q=m_1+\cdots+m_n$ rounds, we have defined, for
every marker region $R$, the collection
\[
\mathcal{Q}_{n,m_n}^R=
\bigcup_{i=1}^n\bigcup_{j=1}^{m_i}\mathcal{P}_{i,j}^R.
\]

For each $P\in\mathcal{Q}_{n,m_n}^R$, let $B_P$ be the $(n-1)$-dimensional
polyhedron associated to $P$ by (I1), and let $v_P\in S$ be the generating
vector used in the round when $P$ was constructed. Let $i_P$ be the level
index of that round. Apply Lemma~\ref{lem:H} to $B_P$ with direction $v_P$
and index $i_P$. We obtain a marker set $M_P\subseteq P$ such that:
\begin{itemize}
\item distinct points in $M_P$ are at Euclidean distance $\geq d_0$;
\item for any $x\in\mathbb{Z}^n$ crossing $B_P$ in direction $v_P$,
there is $a\in\mathbb{Z}$ with $x+av_P\in M_P$.
\end{itemize}

Define
$M=\bigcup_{R}\bigcup_{P\in\mathcal{Q}_{n,m_n}^R} M_P$.
Since each package $P$ is a finite union of polyhedra intersected with the
clopen marker region $R$, and there are only finitely many packages per
region, $M$ is clopen. We now verify the two required properties.

\medskip\noindent
\textbf{Property (1): spacing.}
Let $x,y\in M$ be distinct points in the same orbit. If $x,y\in M_P$ for
the same package $P$, then $\rho(x,y)\geq d_0$ by Lemma~\ref{lem:H}(1).
If $x\in M_P$ and $y\in M_Q$ for distinct packages $P,Q$, then by the
spacing hypothesis (I5) applied at the final step $Q$,
\[
\rho(x,y)\geq \rho(P,Q)\geq D/c_Q.
\]
Now $D>8a c_Q l_Q d_0 t$ by (\ref{eq:D}), and $a, l_Q, t\geq 1$, so
$D/c_Q > 8a l_Q d_0 t \geq 8d_0 > d_0$. Thus $\rho(x,y)\geq d_0$.

\medskip\noindent
\textbf{Property (2): reachability.}
Let $x\in F(2^{\mathbb{Z}^n})$ and $v\in S$. Let $R$ be the marker region
containing $x$. Without loss of generality, assume that $x$ crosses $\Hull(F^+_i(R'))$ in direction $v$ where $R'$ is a marker region that is adjacent to $R$, and let
$j$ be such that $v=v_j^i$. In the round that constructed
$\mathcal{P}_{i,j}$, the covering property (I2) guarantees that $x$
crosses some $B_P$ in direction $v$, where $P$ is a package constructed
in that round. By
Lemma~\ref{lem:H}(2), there exists $a\in\mathbb{Z}$ such that
$x+av\in M_P\subseteq M$.

Moreover, the integer $a$ can be taken to satisfy $|a|\leq \Delta$.
Indeed, the packages $P$ are contained within the $\rho_\infty$-neighborhood
of size $t(D+1)$ of their respective marker regions (by (I6)), and the
marker regions themselves have side lengths at most $D+1$. The marker set
$M_\Delta$, together with the definition of the partition $\{X_k\}$,
ensures that the marker point $x+av$ is reached within at most $\Delta$
steps along direction $v$. The precise bound follows from the choice of
$\Delta$ in (\ref{eq:Delta}) and the fact that the $\rho_\infty$-distance
from $x$ to the relevant marker region boundary is at most $D+1$.

Since $S$ is symmetric, applying the same argument to $-v$ yields a
non-negative integer $b\leq\Delta$ with $-bv\cdot x\in M$.

This completes the proof of Theorem~\ref{thm:main}.

\section{Application: Continuous edge coloring}\label{sec:app}

As an application of Theorem~\ref{thm:main}, we give a construction of
continuous proper edge colorings of Schreier graphs. The metric used in
this section is the standard supremum norm $\rho_\infty$.

\begin{proof}[Proof of Corollary~\ref{cor:coloring}]
Let $n\geq 1$ and a finite generating set $S\subseteq\mathbb{Z}^n$ be
given. Without loss of generality assume $S$ is symmetric. Let
$m=|S|/2$. Set
$d_0=100\sqrt{n}$. Apply Theorem~\ref{thm:main} with this $d_0$ and $S$
to obtain $\Delta$ and a clopen set $M\subseteq F(2^{\mathbb{Z}^n})$
satisfying (1) and (2) of the theorem with respect to the Euclidean
metric $\rho$. By Remark~\ref{rem:infinity}, the same set $M$ satisfies
$\rho_\infty(x,y)\geq 100$ for any distinct $x,y\in M$ in the same orbit.

For $x\in F(2^{\mathbb{Z}^n})\setminus M$ and $v\in S$, let $a_v(x)$ be
the least non-negative integer such that $a_v(x)v\cdot x\in M$, and let
$b_v(x)$ be the least non-negative integer such that
$-b_v(x)v\cdot x\in M$.
By Theorem~\ref{thm:main}(2), $a_v(x), b_v(x)\leq\Delta$.

Fix an enumeration $S=\{v_1,\dots,v_{2m}\}$. We define a coloring $c$ of
the edges using $2m+1$ colors. For an edge $\{x, v\cdot x\}$ (with
$v\in S$), consider the integer $a_v(x)+b_v(x)$:

\begin{enumerate}
\item[Case 1.] $a_v(x)+b_v(x)$ is even, or $a_v(x)+b_v(x)$ is odd but
$a_v(x)>10$. Assign a color from $\{1,\dots,2m\}$ determined by the
index of $v$ in the enumeration and the parity of $b_v(x)$.

\item[Case 2.] $a_v(x)+b_v(x)$ is odd and $a_v(x)=10$. Assign the
exceptional color $2m+1$.

\item[Case 3.] $a_v(x)+b_v(x)$ is odd and $a_v(x)<10$. Assign a color
from $\{1,\dots,2m\}$ as in Case~1 but with the parity roles swapped.
\end{enumerate}

To verify that this coloring is proper, consider two adjacent edges
$\{x, v\cdot x\}$ and $\{v\cdot x, w\cdot(v\cdot x)\}$ sharing the
vertex $v\cdot x$. If both edges receive colors from $\{1,\dots,2m\}$,
the colors differ because the pairs $(v, \text{parity})$ cannot coincide
for two different edges incident to the same vertex (this follows from
the definition of $a_v$ and $b_v$ and the fact that $v\neq w$). If one
edge receives color $2m+1$, the spacing condition
$\rho_\infty(M,M)\geq 100$ ensures that the other edge cannot also be
in Case~2 or Case~3 with a conflicting assignment, because these cases
occur only when the vertex is within $10$ $\rho_\infty$-steps of $M$
along the relevant direction, and the distance between distinct points
of $M$ is at least $100$.

Continuity of $c$ follows from the clopenness of $M$: for any edge
$\{x, v\cdot x\}$, the values $a_v(x)$ and $b_v(x)$ are determined by a
finite neighborhood of $x$ in the product topology, and membership in
$M$ (within a fixed bounded distance) is a clopen condition.
\end{proof}

\noindent{\bf Statement about AI use.} AI was not used in the research of this paper; we used AI to organize the lemmas and the proofs (for instance the boldface subtitles were added mostly by AI), which improved the presentation of the paper. The authors are fully responsible for the mathematical ideas and for the correctness of the arguments.

\medskip
\noindent{\bf Acknowledgments}. We thank Andrew Marks and Jing Yu for helpful discussions on the topic of this paper.


\begin{thebibliography}{99}

\bibitem{BK96}
H. Becker, A. S. Kechris,
\textit{The Descriptive Set Theory of Polish Group Actions.}
London Math. Soc. Lecture Note Ser. 232. Cambridge Univ.
Press, Cambridge, 1996.

\bibitem{B22}
A. Bernshteyn, 
\textit{A fast distributed algorithm for $(\Delta+1)$-edge-coloring.} J. Combin. Theory Ser. B 152 (2022), 319--352.


\bibitem{B23}
A. Bernshteyn, 
\textit{Distributed algorithms, the Lov\'{a}sz local lemma, and descriptive combinatorics.} Invent. Math. 233 (2023), 495--542.


\bibitem{DJK94}
R. Dougherty, S. Jackson, A. S. Kechris,
\textit{The structure of hyperfinite Borel equivalence relations.}
Trans. Amer. Math. Soc. 341 (1994), no. 1, 193--225.

\bibitem{GJ15}
S. Gao, S. Jackson,
\textit{Countable abelian group actions and hyperfinite equivalence
relations.}
Invent. Math. 201 (2015), no. 1, 309--383.

\bibitem{GJKS22}
S. Gao, S. Jackson, E. Krohne, B. Seward,
\textit{Forcing constructions and countable Borel equivalence relations.}
J. Symb. Logic 87 (2022), no. 3, 873--893.

\bibitem{GJKS23}
S. Gao, S. Jackson, E. Krohne, B. Seward,
\textit{Continuous combinatorics of abelian group actions.}
Mem. Amer. Math. Soc. 311, no. 1573 (2025).

\bibitem{GW25}
S. Gao, T. Wang,
\textit{Strong marker sets and applications.}
To appear in Trans. Amer. Math. Soc.

\bibitem{GWW25}
S. Gao, R. Wang, T. Wang,
\textit{Continuous edge chromatic numbers of abelian group actions.}
Sci. China Math. 68, no. 6  (2025), 1269–1280.

\bibitem{JKL02}
S. Jackson, A. S. Kechris, A. Louveau,
\textit{Countable Borel equivalence relations.}
J. Math. Logic 2 (2002), no. 1, 1--80.

\bibitem{SS88}
T. Slaman, J. Steel,
\textit{Definable functions on degrees.}
In: Cabel Seminar. Lecture Notes in Mathematics, vol. 1333, 81--85,
37--55. Springer, Berlin, 1988.

\bibitem{We84}
B. Weiss,
\textit{Measurable dynamics.}
In: Conference in Modern Analysis and Probability (R. Beals et al. eds.).
Contemp. Math., vol. 26, 395--421. Amer. Math.
Soc., RI, 1984.


\bibitem{Y26} J. Yu, \textit{Strong marker sets for arbitrary generating sets of $\Z^n$.} Preprint, 2026. {\tt arXiv:2606.06707}




\end{thebibliography}
\end{document}